\documentclass[a4paper,leqno,10pt]{amsart}

\usepackage{epsf,graphicx,epsfig}
\usepackage{amscd}
\usepackage{amsmath,latexsym,amssymb,amsthm}
\usepackage[nospace,noadjust]{cite}
\usepackage{textcomp}
\usepackage{setspace,cite}
\usepackage{lscape,fancyhdr,fancybox}
\usepackage{stmaryrd}
\usepackage[all,cmtip]{xy}
\usepackage{tikz}
\usepackage{cancel}
\usetikzlibrary{shapes,arrows,decorations.markings}
\usepackage{graphicx}

\usepackage{color}
\usepackage{url}
\usepackage{enumerate}
\usepackage[mathscr]{euscript}

  \theoremstyle{plain}

\swapnumbers
    \newtheorem{thm}{Theorem}[section]
    \newtheorem{prop}[thm]{Proposition}

    \newtheorem{subsec}[thm]{}
\theoremstyle{definition}
    \newtheorem{defn}[thm]{Definition}
        \newtheorem{remark}[thm]{Remark}
    \newtheorem{exam}[thm]{Example}

\theoremstyle{remark}

\title{}
\author{}
\date{}
\usepackage{amssymb}

\usepackage{hyperref}
\hypersetup{
	colorlinks,
	citecolor=blue,
	filecolor=black,
	linkcolor=blue,
	urlcolor=black
}

\begin{document}

\title{Weighted relative Rota-Baxter operators on Leibniz algebras and Post-Leibniz algebra structures}

\author{Apurba Das}
\address{Department of Mathematics,
Indian Institute of Technology, Kharagpur 721302, West Bengal, India.}
\email{apurbadas348@gmail.com}



\begin{abstract}
Leibniz algebras are non-skewsymmetric analogue of Lie algebras.
In this paper, we consider weighted relative Rota-Baxter operators on Leibniz algebras. We define cohomology of such operators and as an application, we study their deformations. Finally, we introduce and study post-Leibniz algebras as the structure behind weighted relative Rota-Baxter operators.
\end{abstract}

\maketitle

\medskip

\medskip

\medskip

\medskip





\medskip

\medskip

\tableofcontents

\noindent
\thispagestyle{empty}


\section{Introduction}\label{sec-1}
Rota-Baxter operators on algebras first appeared in the study of fluctuation theory in probability \cite{rota}. Later, they were found important applications in splitting of algebras \cite{bai-operad}, renormalizations in quantum field theory \cite{connes}, combinatorial study of rooted trees \cite{cartier}, noncommutative analogue of Poisson geometry \cite{uchino}, also in many areas of mathematics and mathematical physics. Relative Rota-Baxter operators are generalization of Rota-Baxter operators in the presence of a representation of the underlying algebra. Recenely, cohomology and deformation theory of relative Rota-Baxter operators are studied in \cite{tang,das-rota}. Relative Rota-Baxter operators on Leibniz algebras, their deformation theory, and relation with Leibniz Yang-Baxter equation, Leibniz bialgebras are extensively considered in \cite{sheng-pl,sheng-defr}. Such operators induce Leibniz-dendriform algebras (also called pre-Leibniz algebras) which are split object for Leibniz algebras.

\medskip

Our aim in this paper is to consider weighted relative Rota-Baxter operators on Leibniz algebras. Let $\mathfrak{g}, \mathfrak{h}$ be two Leibniz algebras and let $\mathfrak{h}$ be a Leibniz $\mathfrak{g}$-representation. A linear map $T : \mathfrak{h} \rightarrow \mathfrak{g}$ is said to be a $\lambda$-weighted relative Rota-Baxter operator if $T$ satisfies some identity that involves linear and quadratic expressions of $T$ (see Definition \ref{defn-wrrb-l}). Here $\lambda \in {\bf k}$ is a fixed scalar. We construct a differential graded Lie algebra whose Maurer-Cartan elements are precisely $\lambda$-weighted relative Rota-Baxter operators. We observe that a $\lambda$-weighted relative Rota-Baxter operator $T$ induces a new Leibniz algebra structure on $\mathfrak{h}$, denoted by $\mathfrak{h}_T$. There is a representation of the Leibniz algebra $\mathfrak{h}_T$ on the vector space $\mathfrak{g}$. The corresponding Leibniz cohomology groups are called the cohomology of $T$. As applications of the cohomology, we study formal and finite order deformations of $T$. We observe that the linear term (coefficients of the formal parameter $t$) in a formal deformation is a $1$-cocycle in the cohomology complex of $T$, whose cohomology class depends only on the equivalence class of the deformation. Given a fnite order deformation of $T$, we associate a second cohomology class which is the obstruction for the extensibility of the deformation.

\medskip

In the last part, we introduce a notion of post-Leibniz algebras. They are non-skewsymmetric analogue of post-Lie algebras introduced by Vallette \cite{vallette}. Further, any pre-Leibniz algebra considered in \cite{sheng-pl} is a particular type of post-Leibniz algebras. Any $\lambda$-weighted relative Rota-Baxter operator $T: \mathfrak{h} \rightarrow \mathfrak{g}$ induces a post-Leibniz algebra structure on $\mathfrak{h}$. We further show that a post-Leibniz algebra gives rise to a Leibniz algebra structure on the underlying vector space, called the total Leibniz algebra. We also find a necessaray and sufficient for a Leibniz algebra to have a compatible post-Leibniz algebra structure.

\medskip

The paper is organized as follows. In Section \ref{sec-2}, we recall some basic definitions and notations regarding Leibniz algebras. We define $\lambda$-weighted relative Rota-Baxter operators and study some basic properties in Section \ref{sec-3}. Cohomology and deformations of $\lambda$-weighted relative Rota-Baxter operators are respectively considered in Section \ref{sec-4} and \ref{sec-5}. Finally, post-Leibniz algebras are defined and studied in Section \ref{sec-6}.

\medskip

All vector spaces, linear and multilinear maps, tensor products are over a field {\bf k} of characteristic $0$. Throughout the paper, $\lambda \in {\bf k}$ is a fixed scalar.

\section{Background on Leibniz algebras}\label{sec-2}
In this section, we recall Leibniz algebras and their cohomology with coefficients in a representation. Our main references are \cite{bala-operad,loday-pira}.

\begin{defn}\label{defn-leibniz-algebra}
A {\bf Leibniz algebra} is a pair $(\mathfrak{g}, [~,~]_\mathfrak{g})$ consists of a vector space $\mathfrak{g}$ together with a bilinear bracket $[~,~]_\mathfrak{g} : \mathfrak{g} \otimes \mathfrak{g} \rightarrow \mathfrak{g}$ satisfying the following Leibniz identity
\begin{align}\label{leibniz-identity}
[x, [ y, z]_\mathfrak{g} ]_\mathfrak{g} = [[x,y]_\mathfrak{g}, z ]_\mathfrak{g} + [ y, [x, z]_\mathfrak{g} ]_\mathfrak{g}, ~ \text{ for } x, y, z \in \mathfrak{g}.
\end{align}
Sometimes we denote a Leibniz algebra as above simply by $\mathfrak{g}$ when the bracket is understood.
\end{defn}

\begin{remark}
(i) The identity (\ref{leibniz-identity}) is equivalent to the fact that the left multiplication map $[x, -]_\mathfrak{g}$ is a derivation for the bracket $[~,~]_\mathfrak{g}$ on $\mathfrak{g}$. Thus, the Leibniz algebras defined in Definition \ref{defn-leibniz-algebra} are also called left Leibniz algebras. The definition of right Leibniz algebras can be similarly defined. Throughout this paper, we shall always work on left Leibniz algebras.

(ii) Any Lie algebra is always a Leibniz algebra. In the presence of skew-symmetry, the Jacobi identity is equivalent to the identity (\ref{leibniz-identity}). Thus, Leibniz algebras can be thought of as a non-skewsymmetric analogue of Lie algebras.
\end{remark}

\begin{defn}
Let $(\mathfrak{g}, [~,~]_\mathfrak{g})$ be a Leibniz algebra. A {\bf representation} of it consists of a triple $(V, \rho^L, \rho^R)$ of a vector space $V$ together with two bilinear maps $\rho^L : \mathfrak{g} \otimes V \rightarrow V$ and $\rho^R : V \otimes \mathfrak{g} \rightarrow V$ (called the left and right $\mathfrak{g}$-actions)satisfying for $x, y \in \mathfrak{g}$, $v \in V$,
\begin{align}
\rho^L (x , \rho^L (y, v)) =~& \rho^L ([x,y]_\mathfrak{g}, v) + \rho^L (y, \rho^L (x, v)), \label{rep1}\\
\rho^L (x, \rho^R (v, y)) =~& \rho^R (\rho^L (x, v), y) + \rho^R (v, [x,y]_\mathfrak{g}), \label{rep2}\\
\rho^R (v, [x,y]_\mathfrak{g}) =~& \rho^R (\rho^R (v,x), y) + \rho^L (x, \rho^R (v, y)). \label{rep3}
\end{align} 
\end{defn}

We often denote a representation simply by $V$ when the bilinear maps $\rho^L$ and $\rho^R$ are clear from the context. It follows from the above definition that any Leibniz algebra $(\mathfrak{g}, [~,~]_\mathfrak{g})$ can be considered as a representation of itself with left and right $\mathfrak{g}$-actions are given by $\rho^L (x, y) = \rho^R (x, y) = [x, y]_\mathfrak{g}$, for $x, y \in \mathfrak{g}$. This is called the adjoint representation.

We will now recall the cohomology of a Leibniz algebra with coefficients in a representation. Let $(\mathfrak{g}, [~,~]_\mathfrak{g})$ be a Leibniz algebra and $(V, \rho^L, \rho^R)$ be a representation of it. For each $n \geq 0$, we define the $n$-th cochain group as $C^n_\mathrm{Leib} (\mathfrak{g}, V) := \mathrm{Hom}(\mathfrak{g}^{\otimes n}, V)$ and the differential $\delta_\mathrm{Leib} : C^n_\mathrm{Leib} (\mathfrak{g}, V) \rightarrow C^{n+1}_\mathrm{Leib} (\mathfrak{g}, V)$ by
\begin{align}
(\delta_{\mathrm{Leib}} f ) (x_1, \ldots, x_{n+1}) =~& \sum_{i=1}^n (-1)^{i+1}~ \rho^L \big( x_i, f(x_1, \ldots, \widehat{x_i}, \ldots, x_{n+1}) \big) + (-1)^{n+1} ~\rho^R \big( f (x_1, \ldots, x_n) , x_{n+1} \big) \\
&+ \sum_{1 \leq i < j \leq n+1} (-1)^i ~ f ( x_1, \ldots, \widehat{x_i}, \ldots, x_{j-1}, [x_i, x_j]_\mathfrak{g}, x_{j+1}, \ldots, x_{n+1} ), \nonumber
\end{align}
for $x_1, \ldots, x_{n+1} \in \mathfrak{g}$. Then $\{ C^\ast_\mathrm{Leib} (\mathfrak{g}, V), \delta_\mathrm{Leib} \}$ is a cochain complex. The corresponding cohomology groups are called the cohomology of the Leibniz algebra $\mathfrak{g}$ with coefficients in the representation $V$.

\medskip

In the following, we recall the Balavoine's bracket that characterizes Leibniz algebras as its Maurer-Cartan elements. We will use this bracket in Section \ref{sec-4} in the Maurer-Cartan characterization of weighted relative Rota-Baxter operators.

Let $\mathfrak{g}$ be a vector space (not necessarily a Leibniz algebra). Let $L = \oplus_{n \geq 0} \mathrm{Hom}(\mathfrak{g}^{\otimes n + 1}, \mathfrak{g})$ be the graded vector space of all multilinear maps on $\mathfrak{g}$ with a degree shift. It carries a graded Lie bracket (the {\bf Balavoine's bracket}) given by
\begin{align*}
\llbracket f, g \rrbracket_B = \sum_{i=1}^{m+1} (-1)^{(i-1)n} f \circ_i g   - (-1)^{mn} \sum_{i=1}^{n+1} (-1)^{(i-1)m} g \circ_i f, ~ \text{ where }
\end{align*}
\begin{align*}
(f \circ_i g)(x_1, \ldots, x_{m+n+1}) =  \sum_{\sigma \in Sh (i-1,n)} (-1)^\sigma f \big( x_{\sigma (1)}, \ldots, x_{\sigma (i-1)}, g(x_{\sigma (i)}, \ldots, x_{\sigma (i+n-1)}, x_{i+n}), \ldots, x_{m+n+1} \big),
\end{align*}
for $f \in \mathrm{Hom}(\mathfrak{g}^{\otimes m+1}, \mathfrak{g})$ and $g \in \mathrm{Hom}(\mathfrak{g}^{\otimes n+1}, \mathfrak{g})$. In other words, $(L, \llbracket ~, ~ \rrbracket_B )$ is a graded Lie algebra. A bilinear bracket $\mu_\mathfrak{g} = [~,~]_\mathfrak{g} \in \mathrm{Hom}(\mathfrak{g}^{\otimes 2}, \mathfrak{g})$ defines a Leibniz algebra structure on $\mathfrak{g}$ if and only if $\mu_\mathfrak{g}$ satisfies $\llbracket \mu_\mathfrak{g}, \mu_\mathfrak{g} \rrbracket_B =0$, i.e., $\mu_\mathfrak{g}$ is a Maurer-Cartan element in the graded Lie algebra $(L, \llbracket ~, ~ \rrbracket_B )$.

\section{Weighted relative Rota-Baxter operators}\label{sec-3}
In this section, we introduce $\lambda$-weighted (relative) Rota-Baxter operators on Leibniz algebras and study some basic properties.

\begin{defn}
Let $\mathfrak{g}$ be a Leibniz algebra. A linear map $T : \mathfrak{g} \rightarrow \mathfrak{g}$ is said to be a {\bf $\lambda$-weighted Rota-Baxter operator} on $\mathfrak{g}$ if the linear map $T$ satisfies
\begin{align*}
[Tx, Ty]_\mathfrak{g} =  T ([Tx,y]_\mathfrak{g} + [x, Ty]_\mathfrak{g} + \lambda [x,y]_\mathfrak{g}),~ \text{ for } x, y \in \mathfrak{g}.
\end{align*}
\end{defn}

\begin{exam}\label{exam-wrb}
(i) A Rota-Baxter operator on $\mathfrak{g}$ is simply a $0$-weighted Rota-Baxter operator.

(ii) The identity map $\mathrm{id} : \mathfrak{g} \rightarrow \mathfrak{g}$ is a $(-1)$-weighted Rota-Baxter operator.

(iii) If $f : \mathfrak{g} \rightarrow \mathfrak{g}$ is a Leibniz algebra morphism and $f^2 = f$ (idempotent condition), then $f$ is a $(-1)$-weighted Rota-Baxter operator. 

(iv) If $T$ is a $\lambda$-weighted Rota-Baxter operator, then for any $\nu \in {\bf k}$, the map $\nu T$ is a $(\nu \lambda)$-weighted Rota-Baxter operator.

(v) If $T$ is a $\lambda$-weighted Rota-Baxter operator, then $- \lambda \mathrm{id} - T$ is a $\lambda$-weighted Rota-Baxter operator.
\end{exam}

\medskip

Let $\mathfrak{g}$ and $\mathfrak{h}$ be two Leibniz algebras. We say that $\mathfrak{h}$ is a {\bf Leibniz $\mathfrak{g}$-representation} if there are bilinear maps $\rho^L : \mathfrak{g} \otimes \mathfrak{h} \rightarrow \mathfrak{h}$ and $\rho^R : \mathfrak{h} \otimes \mathfrak{g} \rightarrow \mathfrak{h}$ that make $(\mathfrak{h}, \rho^L, \rho^R)$ into a representation of the Leibniz algebra $\mathfrak{g}$ satisfying additionally for $u, v \in \mathfrak{h}$, $x \in \mathfrak{g}$,
\begin{align}
[u, \rho^R (v,x)]_\mathfrak{h} =~& \rho^R ([u,v]_\mathfrak{h}, x) + [v, \rho^R (u,x)]_\mathfrak{h}, \label{lrep-1}\\
[u, \rho^L (x,v)]_\mathfrak{h} =~& [\rho^R (u,x), v]_\mathfrak{h} + \rho^L (x, [u, v]_\mathfrak{h}),\label{lrep-2}\\
\rho^L (x, [u, v]_\mathfrak{h}) =~& [\rho^L (x,u), v]_\mathfrak{h} + [u, \rho^L (x,v)]_\mathfrak{h}. \label{lrep-3}
\end{align}
Note that, for any Leibniz algebra $\mathfrak{g}$, the adjoint representation is a Leibniz $\mathfrak{g}$-representation.

\begin{defn}\label{defn-wrrb-l}
Let $\mathfrak{g}$ and $\mathfrak{h}$ be two Leibniz algebras and $\mathfrak{h}$ be a Leibniz $\mathfrak{g}$-representation. A linear map $T : \mathfrak{h} \rightarrow \mathfrak{g}$ is said to be a {\bf $\lambda$-weighted relative Rota-Baxter operator} (on $\mathfrak{h}$ over the Leibniz algebra $\mathfrak{g}$) if
\begin{align}\label{lw}
[Tu, Tv]_\mathfrak{g} = T ( \rho^L(Tu, v) + \rho^R (u, Tv) + \lambda [u,v]_\mathfrak{h}), ~ \text{ for } u, v \in \mathfrak{h}.
\end{align}
\end{defn}

It follows that a $\lambda$-weighted Rota-Baxter operator on a Leibniz algebra $\mathfrak{g}$ is a $\lambda$-weighted relative Rota-Baxter operator on $\mathfrak{g}$ over itself.



Besides Example \ref{exam-wrb}, we have the following examples.

\begin{exam}
(i) Let $\mathfrak{g}, \mathfrak{h}$ be two Leibniz algebras and $\mathfrak{h}$ be a Leibniz $\mathfrak{g}$-representation. A linear map $D : \mathfrak{g} \rightarrow \mathfrak{h}$ is a said to be a $\lambda$-weighted crossed homomorphism (on $\mathfrak{g}$ with values in $\mathfrak{h}$) if
\begin{align*}
D ([x,y]_\mathfrak{g}) = \rho^L (x, Dy) + \rho^R (Dx,y) + \lambda [Dx, Dy]_\mathfrak{h}, ~ \text{ for } x, y \in \mathfrak{g}.
\end{align*} 
If $D$ is an invertible $\lambda$-weighted crossed homomorphism then $D^{-1} : \mathfrak{h} \rightarrow \mathfrak{g}$ is a $\lambda$-weighted relative Rota-Baxter operator.

(ii) Let $\mathfrak{g}$ be a Leibniz algebra. Any ideal $\mathfrak{h}$ of the Leibniz algebra $\mathfrak{g}$ can be regarded as a Leibniz $\mathfrak{g}$-representation. Then the inclusion map $i : \mathfrak{h} \rightarrow \mathfrak{g}$ is a $(-1)$-weighted relative Rota-Baxter operator.
\end{exam}

\begin{prop}
(Weighted semidirect product) Let $\mathfrak{g}$ and $\mathfrak{h}$ be two Leibniz algebras and $\mathfrak{h}$ be a Leibniz $\mathfrak{g}$-representation. Then for any $\lambda \in {\bf k}$, the direct sum $\mathfrak{g} \oplus \mathfrak{h}$ carries a new Leibniz algebra structure with bracket
\begin{align*}
[(x,u), (y, v)]_{\ltimes_\lambda} = \big( [x, y]_\mathfrak{g}, ~ \rho^L (x, v) + \rho^R (u, y) + \lambda [u, v]_\mathfrak{h}  \big),
\end{align*}
for $(x,u), (y, v) \in \mathfrak{g} \oplus \mathfrak{h}$. This is called the $\lambda$-weighted semidirect product, often denoted by $\mathfrak{g} \ltimes_\lambda \mathfrak{h}$.
\end{prop}

We have the following characterization of $\lambda$-weighted relative Rota-Baxter operators.

\begin{prop}
A linear map $T : \mathfrak{h} \rightarrow \mathfrak{g}$ is a $\lambda$-weighted relative Rota-Baxter operator if and only if the graph
$Gr (T) = \{ (Tu, u) |~ u \in \mathfrak{h}  \}$ is a subalgebra of the $\lambda$-weighted semidirect product $\mathfrak{g} \ltimes_\lambda \mathfrak{h}$.
\end{prop}

\begin{proof}
For any $u, v \in \mathfrak{h}$, we have
\begin{align*}
[(Tu,u), (Tv, v)]_{\ltimes_\lambda} = ([Tu, Tv]_\mathfrak{g},~ \rho^L(Tu,u) + \rho^R (u, Tv) + \lambda [u,v]_\mathfrak{h}).
\end{align*}
This is in $Gr (T)$ if and only if (\ref{lw}) holds. Hence the result follows.
\end{proof}

\begin{remark}\label{rmk-induced-leib}
Let $T : \mathfrak{h} \rightarrow \mathfrak{g}$ be a $\lambda$-weighted relative Rota-Baxter operator. Since $Gr (T)$ is isomorphism to $\mathfrak{h}$ as a vector space, we get that $\mathfrak{h}$ inherits a new Leibniz algebra structure with the bracket
\begin{align*}
[u,v]_T : = \rho^L (Tu,v) + \rho^R (u, Tv) + \lambda [u,v]_\mathfrak{h}, ~ \text{ for } u, v \in \mathfrak{h}.
\end{align*}
In other words, $(\mathfrak{h}, [~,~]_T)$ is a Leibniz algebra, denoted by $\mathfrak{h}_T$ (called the induced Leibniz algebra). Moreover, $T: \mathfrak{h}_T \rightarrow \mathfrak{g}$ is a morphism of Leibniz algebras.
\end{remark}

Let $T, T' : \mathfrak{h} \rightarrow \mathfrak{g}$ be two $\lambda$-weighted relative Rota-Baxter operators. A {\bf morphism} from $T$ to $T'$ consists of a pair $(\phi,\psi)$ of Leibniz algebra morphisms $\phi : \mathfrak{g} \rightarrow \mathfrak{g}$ and $\psi : \mathfrak{h} \rightarrow \mathfrak{h}$ satisfying
\begin{align*}
\phi \circ T = T' \circ \psi, \quad \psi (\rho^L (x, u) ) = \rho^L (\phi (x), \psi (u)) ~~~~ \text{ and } ~~~~ \psi (\rho^R (u, x)) = \rho^R (\psi (u), \phi (x)), ~ \text{ for } x \in \mathfrak{g}, u \in \mathfrak{h}.
\end{align*}

\begin{prop}
Let $(\phi, \psi)$ be a morphism of $\lambda$-weighted relative Rota-Baxter operators from $T$ to $T'$. Then $\psi : \mathfrak{h} \rightarrow \mathfrak{h}$ is a morphism of induced Leibniz algebras from $(\mathfrak{h}, [~,~]_T)$ to $(\mathfrak{h}, [~,~]_{T'})$.
\end{prop}

\begin{proof}
For any $u, v \in \mathfrak{h}$, we have
\begin{align*}
\psi ([u,v]_T) =~& \psi \big( \rho^L (Tu,v) + \rho^R (u, Tv) + \lambda [u,v]_\mathfrak{h}  \big) \\
=~& \rho^L (\phi \circ T (u), \psi (v)) + \rho^R (\psi (u), \phi \circ T (v)) + \lambda [\psi (u), \psi (v)]_\mathfrak{h} \\
=~& \rho^L (T' (\psi (u)), \psi (v))  + \rho^R (\psi (u), T' (\psi (v))) + \lambda [\psi (u), \psi (v)]_\mathfrak{h} = [ \psi (u), \psi (v) ]_{T'}.
\end{align*}
This shows that $\psi : (\mathfrak{h}, [~,~]_T) \rightarrow (\mathfrak{h}, [~,~]_{T'})$ is a morphism of Leibniz algebras.
\end{proof}

\section{Maurer-Cartan characterization and cohomology of weighted relative Rota-Baxter operators}\label{sec-4}
Let $\mathfrak{g}$ and $\mathfrak{h}$ be two Leibniz algebras and $\mathfrak{h}$ be a Leibniz $\mathfrak{g}$-representation. With this data, we first construct a differential graded Lie algebra whose Maurer-Cartan elements are precisely $\lambda$-weighted relative Rota-Baxter operators. This characterization allows us to define a cohomology associated to a $\lambda$-weighted relative Rota-Baxter operator. Later, we interpret this as the cohomology of the corresponding induced Leibniz algebra with coefficients in a suitable representation.

Let $\mu_\mathfrak{g} \in \mathrm{Hom} (\mathfrak{g}^{\otimes 2}, \mathfrak{g})$ and $\mu_\mathfrak{h} \in \mathrm{Hom} (\mathfrak{h}^{\otimes 2}, \mathfrak{h})$ denote the Leibniz operations on $\mathfrak{g}$ and $\mathfrak{h}$, respectively. Moreover, let $\rho^L : \mathfrak{g} \otimes \mathfrak{h} \rightarrow \mathfrak{h}$ and $\rho^R : \mathfrak{h} \otimes \mathfrak{g} \rightarrow \mathfrak{h}$ denote the left and right $\mathfrak{g}$-actions on $\mathfrak{h}$, respectively.  Take $V = \mathfrak{g} \oplus \mathfrak{h}$ and consider the graded Lie algebra $(L = \oplus_{n \geq 0} \mathrm{Hom}(V^{\otimes n+1}, V), [~,~]_B)$ on the shifted space of multilinear maps on $V$ with the Balavoine's bracket. Then it can be checked that the element $\theta := \mu_\mathfrak{g} + \rho^L + \rho^R \in \mathrm{Hom}(V^{\otimes 2}, V)$ defined by
\begin{align*}
\theta ((x,u), (y,v))= ([x,y]_\mathfrak{g},~ \rho^L (x,v) + \rho^R (u,y)),~ \text{ for } (x,y), (y,v) \in V
\end{align*}
satisfies $[\theta, \theta]_B = 0$. Hence $\theta$ induces a degree $1$ differential $d_\theta = [\theta, -]_B$ on the graded vector space $L$. Moreover, the graded subspace $\mathfrak{a} = \oplus_{n \geq 0} \mathrm{Hom}(\mathfrak{h}^{\otimes n+1}, \mathfrak{g}) \subset L$ is an abelian Lie subalgebra. Hence, by the derived bracket construction \cite{voro}, the shifted space $\mathfrak{a}[-1] = \oplus_{n \geq 1} \mathrm{Hom}(\mathfrak{h}^{\otimes n}, \mathfrak{g})$ carries a graded Lie bracket given by
\begin{align*}
\llbracket P, Q \rrbracket := (-1)^m [d_\theta (P), Q]_B = (-1)^m [[\theta, P]_B, Q]_B,
\end{align*}
for $P \in \mathrm{Hom}(\mathfrak{h}^{\otimes m}, \mathfrak{g}), Q \in \mathrm{Hom}(\mathfrak{h}^{\otimes n}, \mathfrak{g}).$ The explicit description of the bracket is as follows (see \cite{sheng-pl})
\begin{align}\label{b-formula}
&\llbracket P, Q \rrbracket (u_1, \ldots, u_{m+n}) \\
&= \sum_{i=1}^m \sum_{\sigma \in Sh (i-1,n)} (-1)^{(i-1)n } (-1)^\sigma P \big(  u_{\sigma (1)}, \ldots, u_{\sigma (i-1)}, \rho^L ( Q (u_{\sigma (i)}, \ldots, u_{\sigma (i+n-1)}), u_{i+n}), \ldots, u_{m+n} \big) \nonumber \\
&+ \sum_{i=1}^m \sum_{ {\sigma \in Sh (i-1,1, n-1) } } (-1)^{(i-1)n} (-1)^\sigma P \big(  u_{\sigma (1)}, \ldots, \rho^R ( u_{\sigma (i)}, Q (  u_{\sigma (i+1)}, \ldots, u_{\sigma (i+n-1)}, u_{i+n} )), \ldots, u_{m+n}    \big)  \nonumber \\
&+ \sum_{\sigma \in Sh (m,n-1)} (-1)^{mn} (-1)^\sigma [ P ( u_{\sigma (1)}, \ldots, u_{\sigma (m)} ), Q (u_{\sigma (m+1)}, \ldots, u_{\sigma (m+n-1)}, u_{m+n})]_\mathfrak{g} \nonumber \\
&- (-1)^{mn} \big\{ \sum_{i=1}^n \sum_{\sigma \in Sh (i-1,m)} (-1)^{(i-1)m } (-1)^\sigma Q \big(  u_{\sigma (1)}, \ldots, \rho^L ( P (u_{\sigma (i)}, \ldots, u_{\sigma (i+m-1)}), u_{i+m}), \ldots, u_{m+n} \big) \nonumber  \\
&+ \sum_{i=1}^n \sum_{ \sigma \in Sh (i-1,1,m-1) } (-1)^{(i-1)m} (-1)^\sigma Q \big(  u_{\sigma (1)}, \ldots, \rho^R ( u_{\sigma (i)}, P (  u_{\sigma (i+1)}, \ldots, u_{\sigma (i+m-1)}, u_{i+m} )), \ldots, u_{m+n}    \big)  \nonumber \\
&+ \sum_{\sigma \in Sh (n,m-1)} (-1)^{mn} (-1)^\sigma [ Q ( u_{\sigma (1)}, \ldots, u_{\sigma (n)} ), P (u_{\sigma (n+1)}, \ldots, u_{\sigma (m+n-1)}, u_{m+n})]_\mathfrak{g}  \big\}.  \nonumber 
\end{align}

On the other hand, for any $\lambda \in {\bf k}$, the element $\theta' =  -\lambda \mu_\mathfrak{h} \in \mathrm{Hom}(V^{\otimes 2}, V)$ defined by
\begin{align*}
\theta' ((x,u), (y,v)) = (0,  -\lambda [u, v]_\mathfrak{h}),~ \text{for } (x, u), (y,v) \in V
\end{align*}
satisfies $[\theta', \theta']_B = 0$. Hence the element $\theta'$ also induces a differential $d_{\theta'} =[\theta', -]_B$ on the graded vector space $L$. It is then easy to observe that the graded subspace $\mathfrak{a}[-1] = \oplus_{n \geq 1} (\mathfrak{h}^{\otimes n}, \mathfrak{g})$ is closed under the differential $d_{\theta'}$. Let us denote the restriction of $d_{\theta'}$ to this subspace by $d$. This is explicitly given by
\begin{align}\label{d-formula}
(dP)(u_1, \ldots, u_{n+1}) = (-1)^{n} \sum_{1 \leq i < j \leq n+1} (-1)^i P (u_1, \ldots, \widehat{u_i}, \ldots, u_{j-1}, \lambda [u_i, u_j]_\mathfrak{h}, u_{j+1}, \ldots, u_{n+1}).
\end{align}

\medskip

Finally, we observe that the elements $\theta$ and $\theta'$ additionally satisfies the compatibility $[\theta, \theta']_B =0$. This implies that
\begin{align*}
d \llbracket P, Q \rrbracket 
=& (-1)^m [\theta', [[\theta, P]_B, Q]_B ]_B \\
=& (-1)^m [[\theta', [\theta, P]_B ]_B, Q ]_B + (-1)^m (-1)^m [[\theta, P]_B, [\theta', Q]_B ]_B \\
=& (-1)^{m+1} [[[ \theta, [\theta',P]_B ]_B, Q ]_B + [[\theta, P]_B, dQ]_B \\
=& (-1)^{m+1} [[\theta,dP]_B, Q]_B +  [[\theta, P]_B, dQ]_B = \llbracket dP, Q \rrbracket + (-1)^m \llbracket P, dQ \rrbracket.
\end{align*}
Thus, we get the following.

\begin{thm}\label{mc-thm}
Let $\mathfrak{g}, \mathfrak{h}$ be two Leibniz algebras and $\mathfrak{h}$ be a Leibniz $\mathfrak{g}$-representation. Then the triple $(\oplus_{n \geq 1} \mathrm{Hom}(\mathfrak{h}^{\otimes n}, \mathfrak{g}), \llbracket ~, ~ \rrbracket, d)$ is a differential graded Lie algebra. Moreover, a linear map $T : \mathfrak{h} \rightarrow \mathfrak{g}$ is a $\lambda$-weighted relative Rota-Baxter operator if and only if $T$ is a Maurer-Cartan element in the differential graded Lie algebra $(\oplus_{n \geq 1} \mathrm{Hom}(\mathfrak{h}^{\otimes n}, \mathfrak{g}), \llbracket ~, ~ \rrbracket, d)$.
\end{thm}

\begin{proof}
The first part follows from previous discussions. For the second part, we observe that for any linear map $T : \mathfrak{h} \rightarrow \mathfrak{g}$, we have
\begin{align*}
\llbracket T, T \rrbracket (u,v) =~& 2 \big(  T (\rho^L (Tu,v)) + T (\rho^R (u, Tv) - [Tu,Tv]_\mathfrak{g} \big) \qquad (\text{from } (\ref{b-formula})) \\
(dT)(u,v) =~&  \lambda T ([u,v]_\mathfrak{h}) \qquad (\text{from }\ref{d-formula})).
\end{align*}
Therefore,
\begin{align*}
(dT + \frac{1}{2}\llbracket T, T \rrbracket)(u,v) = 2 \big( T (\rho^L (Tu,v) +  \rho^R (u, Tv) + \lambda [u, v]_\mathfrak{h}) - [Tu,Tv]_\mathfrak{g}  \big). 
\end{align*}
This shows that $T$ is a Maurer-Cartan element (i.e. $dT + \frac{1}{2}\llbracket T, T \rrbracket = 0$) if and only if $T$ is a $\lambda$-weighted relative Rota-Baxter operator.
\end{proof}

\medskip

Let $T : \mathfrak{h} \rightarrow \mathfrak{g}$ be a $\lambda$-weighted relative Rota-Baxter operator. That is, $T$ can be equivalently seen as a Maurer-Cartan element in the differential graded Lie algebra constructed in Theorem \ref{mc-thm}. Therefore, $T$ induces a differential
\begin{align*}
d_T = d + \llbracket T, - \rrbracket : \mathrm{Hom}(\mathfrak{h}^{\otimes n}, \mathfrak{g}) \rightarrow \mathrm{Hom}(\mathfrak{h}^{\otimes n+1}, \mathfrak{g}), ~ \text{ for } n \geq 1.
\end{align*}
Explicitly, we have
\begin{align}\label{dt-formula}
&(d_T f) (u_1, \ldots, u_{n+1}) \\
&= (-1)^n \sum_{i=1}^n (-1)^{i+1} \big(  [Tu_i, f (u_1, \ldots, \widehat{u_i}, \ldots, u_{n+1})]_\mathfrak{g} - T (\rho^R (u_i, f (u_1, \ldots, \widehat{u_i}, \ldots, u_{n+1}) )) \big) \nonumber \\
&- \big( [f (u_1, \ldots, u_n), T(u_{n+1})]_\mathfrak{g} - T (\rho^L ( f (u_1, \ldots, u_n), u_{n+1}  ))  \big) \nonumber \\
&+ \sum_{1 \leq i < j \leq n+1} (-1)^{i+n} f (u_1, \ldots, \widehat{u_i}, \ldots, u_{j-1}, \rho^L (Tu, v)+ \rho^R (u, Tv) + \lambda [u,v]_\mathfrak{h}, u_{j+1}, \ldots, u_{n+1}). \nonumber
\end{align}


\medskip

In the following, we show that the differential (\ref{dt-formula}) can be seen as the differential of the Leibniz algebra cohomology of $\mathfrak{h}_T$ (see Remark \ref{rmk-induced-leib}) with coefficients in a suitable representation. We first observe the following.

\begin{prop}
Let $T : \mathfrak{h} \rightarrow \mathfrak{g}$ be a $\lambda$-weighted relative Rota-Baxter operator. Then $(\mathfrak{g}, \varrho^L, \varrho^R)$ is a representation of the induced Leibniz algebra $\mathfrak{h}_T$ with left and right $\mathfrak{h}_T$-actions given by
\begin{align*}
\varrho^L : \mathfrak{h}_T \otimes \mathfrak{g} \rightarrow \mathfrak{g}, \quad \varrho^L (u,x) = [Tu,x]_\mathfrak{g} - T (\rho^R (u,x)),\\
\varrho^R : \mathfrak{g} \otimes \mathfrak{h}_T \rightarrow \mathfrak{g}, \quad \varrho^R (x,u) = [x, Tu]_\mathfrak{g} - T (\rho^L (x,u)).
\end{align*}
\end{prop}

\begin{proof}
For any $u, v \in \mathfrak{h}$ and $x \in \mathfrak{g}$,
\begin{align*}
&\varrho^L (u, \varrho^L (v,x)) - \varrho^L ([u,v]_T,x) - \varrho^L (v, \varrho^L (u,x)) \\
&= \varrho^L \big( u, [Tv, x]_\mathfrak{g} - T\rho^R (v,x) \big) - [T[u,v]_T, x]_\mathfrak{g} + T \rho^R ([u,v]_T,x) - \varrho^L (v, [Tu,x]_\mathfrak{g} - T\rho^R (u,x)) \\
&= \cancel{[Tu, [Tv, x]_\mathfrak{g} ]_\mathfrak{g}} - [Tu, T \rho^R (v,x)]_\mathfrak{g} - T \rho^R (u, [Tv,x]_\mathfrak{g}) + T \rho^R (u, T \rho^R (v,x)) \\
&- \cancel{[[Tu, Tv]_\mathfrak{g}, x]_\mathfrak{g}  }
+ T \rho^R \big(  \rho^L (Tu, v) + \rho^R (u, Tv) + \lambda [u, v]_\mathfrak{h}, x \big) \\
&- \cancel{[Tv, [Tu,x]_\mathfrak{g} ]_\mathfrak{g}} + [Tv, T \rho^R (u,x) ]_\mathfrak{g} + T \rho^R (v, [Tu, x]_\mathfrak{g}) - T \rho^L (v, T \rho^R (u,x)) \\
&= - T \big( \rho^L (Tu, \rho^R (v,x)) + \rho^R (u, T \rho^R (v,x)) + \lambda [u, \rho^R (v,x)]_\mathfrak{h}   \big) + T \rho^R \big(  \rho^L (Tu, v) + \rho^R (u, Tv) + \lambda [u,v]_\mathfrak{h}, x \big)\\
&+ T \big( \rho^L (Tv, \rho^R (u,x)) + \rho^R (v, T\rho^R (u,x)) + \lambda [v, \rho^R (u,x)]_\mathfrak{h}  \big) + T (\rho^R (v, [Tu,x]_\mathfrak{g})) - T (\rho^R (v, T \rho^R (u,x))) \\
&= 0.
\end{align*}
Similarly, 
\begin{align*}
&\varrho^L (u, \varrho^R (x, v)) - \varrho^R (\varrho^L (u,x), v) - \varrho^R (x, [u,v]_T) \\
&= \varrho^L (u, [x, Tv]_\mathfrak{g} - T \rho^L (x, v)) -\varrho^R ([Tu,x]_\mathfrak{g} - T \rho^R (u,x), v) - [x, T [u,v]_T ]_\mathfrak{g} + T \rho^L (x, [u, v]_T) \\
&= \cancel{[Tu, [x, Tv]_\mathfrak{g} ]_\mathfrak{g}} - [Tu, T \rho^L (x, v)]_\mathfrak{g} - T \rho^R (u, [x, Tv]_\mathfrak{g}) + T \rho^R (u, T \rho^L (x, v)) \\
&- \cancel{[[Tu, x]_\mathfrak{g}, Tv]_\mathfrak{g}} + T \rho^L ([Tu, x]_\mathfrak{g}, v)  + [ T \rho^R (u, x), Tv]_\mathfrak{g} - T \rho^L (T \rho^R (u,x), v) \\
&- \cancel{[x, [Tu, Tv]_\mathfrak{g} ]_\mathfrak{g}} + T \rho^L \big( x, \rho^L (Tu, v)) + \rho^R (u, Tv) + \lambda [u, v]_\mathfrak{h} \big) \\
&= - T \big(  \rho^L (Tu, \rho^L (x,v))  + \rho^R (u, T \rho^L (x,v)) + \lambda [u, \rho^L (x,v)]_\mathfrak{h} \big) - T \rho^R (u, [x, Tv]_\mathfrak{g}) + T \rho^R (u, T \rho^L (x, v)) \\
&+ T \rho^L ([Tu,x]_\mathfrak{g} , v) + T \big( \rho^L (T \rho^R (u,x), v) + \rho^R (\rho^R (u,x), Tv) + \lambda [\rho^R (u,x), v]_\mathfrak{h}  \big) - T \rho^L (T \rho^R (u,x), v) \\
&+ T \rho^L \big(  x, \rho^L (Tu,v) + \rho^R (u, Tv) + \lambda [u,v]_\mathfrak{h} \big) \\
&= 0
\end{align*}
and
\begin{align*}
&\varrho^L (x, [u,v]_T) - \varrho^R (\varrho^R (x, u), v) - \varrho^L (u, \varrho^R (x, v)) \\
&= [x, T[u,v]_T ]_\mathfrak{g} - T \rho^R (x, [u,v]_T) - \varrho^R ([x, Tu]_\mathfrak{g} - T \rho^L (x, u), v) - \varrho^L (u, [x, Tv]_\mathfrak{g} - T \rho^L (x, v)) \\
&= \cancel{[x, [Tu, Tv]_\mathfrak{g} ]_\mathfrak{g}} - T \rho^R \big(  x, \rho^L (Tu,v) + \rho^R (u, Tv) + \lambda [u,v]_\mathfrak{h} \big) \\
&- \cancel{[[x, Tu]_\mathfrak{g}, Tv]_\mathfrak{g}} + T \rho^L ([x, Tu]_\mathfrak{g}, v) + [ T \rho^L (x, u), Tv]_\mathfrak{g} - T \rho^L (T \rho^L (x, u), v) \\
&- \cancel{[Tu, [x, Tv]_\mathfrak{g} ]_\mathfrak{g}} + T \rho^R (u, [x, Tv]_\mathfrak{g}) + [Tu, T \rho^L (x, v) ]_\mathfrak{g} - T \rho^R (u, T \rho^L (x, v)) \\
&= - T \rho^R \big(  x, \rho^L (Tu,v) + \rho^R (u, Tv) + \lambda [u,v]_\mathfrak{h} \big)\\
&+ T \rho^L ([x, Tu]_\mathfrak{g}, v) + T \big( \rho^L (T \rho^L (x, u), v) + \rho^R (\rho^L (x,u), TV) + \lambda [\rho^L (x, u), v]_\mathfrak{h}  \big) - T \rho^L (T \rho^L (x, u), v) \\
&+  T \rho^R (u, [x, Tv]_\mathfrak{g})  + T \big( \rho^L (Tu, \rho^L (x,v)) + \rho^R (u, T \rho^L (x, v)) + \lambda [u, \rho^L (x,v)]_\mathfrak{h}  \big) - T \rho^R (u, T \rho^L (x, v)) \\
&= 0.
\end{align*}
This proves that $(\mathfrak{g}, \varrho^L, \varrho^R)$ is a representation of the Leibniz algebra $\mathfrak{h}_T$.
\end{proof}

Therefore, by the previous proposition, we can consider the cochain complex $\{ C^\ast_\mathrm{Leib} (\mathfrak{h}_T, \mathfrak{g}), \delta_\mathrm{Leib} \}$  of the Leibniz algebra $\mathfrak{h}_T$ with coefficients in the representation $(\mathfrak{g}, \varrho^L, \varrho^R)$. More precisely, $C^n_\mathrm{Leib}(\mathfrak{h}_T, \mathfrak{g}) = \mathrm{Hom}(\mathfrak{h}^{\otimes n}, \mathfrak{g})$, for $n \geq 0$, and the differential $\delta_\mathrm{Leib} : C^n_\mathrm{Leib}(\mathfrak{h}_T, \mathfrak{g}) \rightarrow C^{n+1}_\mathrm{Leib}(\mathfrak{h}_T, \mathfrak{g})$ given by
\begin{align}\label{www-l}
&(\delta_\mathrm{Leib} f) (u_1, \ldots, u_{n+1})\\
&= \sum_{i=1}^n (-1)^{i+1} \big(  [Tu_i, f (u_1, \ldots, \widehat{u_i}, \ldots, u_{n+1})]_\mathfrak{g} - T (\rho^R (u_i, f (u_1, \ldots, \widehat{u_i}, \ldots, u_{n+1}) )) \big) \nonumber \\
&+ (-1)^{n+1} \big( [f (u_1, \ldots, u_n), T(u_{n+1})]_\mathfrak{g} - T (\rho^L ( f (u_1, \ldots, u_n), u_{n+1}  ))  \big) \nonumber \\
&+ \sum_{1 \leq i < j \leq n+1} (-1)^{i} f (u_1, \ldots, \widehat{u_i}, \ldots, u_{j-1}, \rho^L (Tu, v)+ \rho^R (u, Tv) + \lambda [u,v]_\mathfrak{h}, u_{j+1}, \ldots, u_{n+1}). \nonumber
\end{align}
Let $Z^n_\mathrm{Leib} (\mathfrak{h}_T, \mathfrak{g})$ and $B^n_\mathrm{Leib} (\mathfrak{h}_T, \mathfrak{g})$ be the space of $n$-cocycles and $n$-coboundaries, respectively. Then we have $B^n_\mathrm{Leib} (\mathfrak{h}, \mathfrak{g}) \subset Z^n_\mathrm{Leib} (\mathfrak{h}, \mathfrak{g})$. The corresponding quotients
\begin{align*}
H^n_\mathrm{Leib} (\mathfrak{h}_T, \mathfrak{g}) := \frac{ Z^n_\mathrm{Leib} (\mathfrak{h}_T, \mathfrak{g}) }{ B^n_\mathrm{Leib} (\mathfrak{h}_T, \mathfrak{g}) },~ \text{ for } n \geq 0
\end{align*}
are called the cohomology groups of the $\lambda$-weighted relative Rota-Baxter operator $T$.

\begin{remark}
It follows from the expressions of (\ref{dt-formula}) and (\ref{www-l}) that $d_T f = (-1)^{n} \delta_\mathrm{Leib}f$, for $f \in \mathrm{Hom} (\mathfrak{h}^{\otimes n}, \mathfrak{g})$ with $n \geq 1$. Hence the cohomology groups induced by the differential $d_T$ are isomorphic to the cohomology groups of $T$.
\end{remark}

%

\section{Deformations of weighted relative Rota-Baxter operators}\label{sec-5}
In this section, we studey formal and finite order deformations of a $\lambda$-weighted relative Rota-Baxter operator in terms of the cohomology theory introduced in the previous section. In particular, we find a sufficient condition for the rigidity of a $\lambda$-weighted relative Rota-Baxter operator, and find a obstruction class for the extensibility of a finite order deformation.

\medskip

\noindent {\bf Formal deformations.} Let $\mathfrak{g}, \mathfrak{h}$ be two Leibniz algebras and $\mathfrak{h}$ be a Leibniz $\mathfrak{g}$-representation. We consider the spaces $\mathfrak{g}[[t]]$ and $\mathfrak{h}[[t]]$ of formal power series in $t$ with coefficients from $\mathfrak{g}$ and $\mathfrak{h}$, respectively. Note that $\mathfrak{g}[[t]]$ and $\mathfrak{h}[[t]]$ are both ${\bf k}[[t]]$-modules. It is easy to see that the Leibniz algebra structures on $\mathfrak{g}$ and $\mathfrak{h}$ can be extended by ${\bf k}[[t]]$-bilinearity to Leibniz algebra structures on $\mathfrak{g}[[t]]$ and $\mathfrak{h}[[t]]$, respectively. Moreover, the Leibniz $\mathfrak{g}$-representation on $\mathfrak{h}$ also induces a Leibniz $\mathfrak{g}[[t]]$-representation on $\mathfrak{h}[[t]]$.

\begin{defn}
Let $T: \mathfrak{h} \rightarrow \mathfrak{g}$ be a $\lambda$-weighted relative Rota-Baxter operator. A {\bf formal deformation} of $T$ consists of a formal sum $T_t = \sum_{i=0}^\infty t^i T_i \in \mathrm{Hom}(\mathfrak{h}, \mathfrak{g})[[t]]$ with $T_0 = T$, such that the ${\bf k}[[t]]$-linear map $T_t : \mathfrak{h}[[t]] \rightarrow \mathfrak{g}[[t]]$ is a $\lambda$-weighted relative Rota-Baxter operator (on $\mathfrak{h}[[t]]$ over the Leibniz algebra $\mathfrak{g}[[t]]$).
\end{defn}

It follows from the above definition that $T_t$ is a formal deformation of $T$ if and only if
\begin{align}\label{def-eqns}
\sum_{i+j = n} [ T_i (u), T_j (v)]_\mathfrak{g} = \sum_{i+j = n} T_i \big( \rho^L (T_j (u), v) + \rho^R (u, T_j (v)) \big) + \lambda T_n ([u,v]_\mathfrak{h}),
\end{align}
for all $u, v \in \mathfrak{h}$ and $n \geq 0$. The above system of equations are called deformation equations. Note that (\ref{def-eqns}) holds for $n =0$ as $T_0 = T$ is a $\lambda$-weighted relative Rota-Baxter operator. However, for $n =1$, we get that
\begin{align*}
[Tu, T_1 v]_\mathfrak{g} + [T_1 u, Tv]_\mathfrak{g} = T (\rho^L (T_1 u, v) + \rho^R (u, T_1 v)) + T_1 \big( \rho^L (Tu,v) + \rho^R (u, Tv) + \lambda [u,v]_\mathfrak{h}  \big).
\end{align*}
This implies that $\delta_\mathrm{Leib} (T_1) = 0$ (follows from (\ref{www-l})). In other words, $T_1$ is a $1$-cocycle in the cohomology complex of $T$. This is called the infinitesimal of the deformation.

\begin{defn}
Two deformations $T_t$ and $T_t'$ of a $\lambda$-weighted relative Rota-Baxter operator $T$ are said to be equivalent if there exists an element $x_0 \in \mathfrak{g}$ and linear maps $\phi_i \in \mathrm{Hom}(\mathfrak{g}, \mathfrak{g})$, $\psi_i \in \mathrm{Hom}(\mathfrak{h}, \mathfrak{h})$ for $i \geq 2$, such that the pair of maps
\begin{align*}
\big( \Phi_t = \mathrm{id}_\mathfrak{g} + t [x_0, -]_\mathfrak{g} + \sum_{i \geq 2} t^i \phi_i : \mathfrak{g}[[t]] \rightarrow \mathfrak{g}[[t]] ~, ~ \Psi_t = \mathrm{id}_\mathfrak{h} + t \rho^L (x_0, -) + \sum_{i \geq 2} t^i \psi_i  : \mathfrak{h}[[t]] \rightarrow \mathfrak{h} [[t]] \big)
\end{align*}
is a morphism of $\lambda$-weighted relative Rota-Baxter operators from $T_t$ to $T_t'$.
\end{defn}

Therefore, if $T_t$ and $T_t'$ are equivalent, then we must have
\begin{align*}
(\Phi_t \circ T_t) (u) = (T_t' \circ \Psi_t)(u), \quad
\Psi_t (\rho^L (x,u)) = \rho^L (\Phi_t (x), \Psi_t (u)) ~~~~ \text{ and } ~~~~ \Psi_t (\rho^R (u,x)) = \rho^R (\Psi_t (u), \Phi_t (x)),
\end{align*}
for $x \in \mathfrak{g}$, $u \in \mathfrak{h}$. By expanding the identity $(\Phi_t \circ T_t) (u) = (T_t' \circ \Psi_t)(u)$ and equating coefficients of $t$ from both sides, we get
\begin{align*}
T_1 (u) - T_1' (u) = T \rho^L (x_0, u) - [x_0, Tu]_\mathfrak{g} = \delta_\mathrm{Leib} (x_0) (u).
\end{align*}
As a summary of the above discussions, we get the following.

\begin{thm}\label{inf-1co}
Let $T: \mathfrak{h} \rightarrow \mathfrak{g}$ be a $\lambda$-weighted relative Rota-Baxter operator. If $T_t = \sum_{i=0}^\infty t^i T_i$ is a formal deformation of $T$, then $T_1$ is a $1$-cocycle in the cohomology complex of $T$. Moreover, the corresponding cohomology class depends only on the equivalence class of the deformation $T_t$.
\end{thm}

\medskip

\begin{defn}
A $\lambda$-weighted relative Rota-Baxter operator $T$ is said to be rigid if any formal deformation $T_t$ is equivalent to the undeformed one $T_t' = T.$
\end{defn}

Next, we introduce certain specific elements associated to $T$ whose importance will be clear in the next theorem.

\begin{defn}
An element $x_0 \in \mathfrak{g}$ is said to be a {\bf Nijenhuis element} associated with $T$ if 
\begin{align*}
[[x_0, x]_\mathfrak{g}, [x_0, y]_\mathfrak{g}]_\mathfrak{g} = 0, \quad [\rho^L (x_0, u), \rho^L (x_0, v)]_\mathfrak{h} = 0, \\
\rho^L ([x_0, x]_\mathfrak{g}, \rho^L (x_0, u)) = 0, \quad \rho^R (\rho^L (x_0, u), [x_0, x]_\mathfrak{g}) = 0,
\end{align*}
for $x, y \in \mathfrak{g}$ and $u, v \in \mathfrak{h}$. We denote the set of all Nijenhuis elements by $\mathrm{Nij}(T).$ 
\end{defn}

Note that any element lying in the intersection
\begin{align*}
\{ x_0 \in \mathfrak{g} ~|~ [x_0, x]_\mathfrak{g} = 0, \forall x \in \mathfrak{g}  \} \cap \{ x_0 \in \mathfrak{g} ~|~ \rho^L (x_0, u ) = 0, \forall u \in \mathfrak{h} \}
\end{align*}
belongs to $\mathrm{Nij}(T).$ We now have the following interesting result.

\begin{thm}
Let $T: \mathfrak{h} \rightarrow \mathfrak{g}$ be a $\lambda$-weighted relative Rota-Baxter operator. If $Z^1 (\mathfrak{h}_T, \mathfrak{g}) = \delta_\mathrm{Leib} (\mathrm{Nij}(T))$ then $T$ is rigid.
\end{thm}

\begin{proof}
Let $T_t = \sum_{i=0}^\infty t^i T_i$ be a formal deformation of $T$. Then by Theorem \ref{inf-1co}, we get that $T_1 \in Z^1_\mathrm{Leib} (\mathfrak{h}_T, \mathfrak{g})$ is a $1$-cocycle in the cohomology complex of $T$. Hence by hypothesis, there exists an element $x_0 \in \mathrm{Nij}(T)$ such that $T_1 = \delta_\mathrm{Leib} (x_0)$ (i.e. $T_1 (u) = - [x_0, Tu]_\mathfrak{g} + T \rho^L (x_0, u)$ for all $u \in \mathfrak{h}$). Sett $\Phi_t = \mathrm{id}_\mathfrak{g} + t [x_0, -]_\mathfrak{g}$ and $\Psi_t = \mathrm{id}_\mathfrak{h} + t \rho^L (x_0, -)$, and define $T_t' = \Phi_t \circ T_t \circ \Psi_t^{-1}$. Then $T_t'$ is a formal deformation of $T$. Since $x_0 \in \mathrm{Nij}(T)$, it follows that $\Phi_t$ and $\Psi_t$ are Leibniz algebra morphisms, also $\Psi_t (\rho^L (x,u)) = \rho^L (\Phi_t(x), \Psi_t (u))$ and  $\Psi_t (\rho^R (u,x)) = \rho^R ( \Psi_t (u), \Phi_t (x))$ hold. Hence $T_t'$ is equivalent to $T_t$. We also have
\begin{align*}
T_t' (u) =~& (\mathrm{id}_\mathfrak{g} + t [x_0, -]_\mathfrak{g}) \circ (\sum_{i=0}^\infty t^i T_i) (u - t \rho^L (x_0, u) + \text{powers of } t^{\geq 2}) \\
=~& (\mathrm{id}_\mathfrak{g} + t [x_0, -]_\mathfrak{g}) \big( Tu - t~ T \rho^L (x_0, u) + t T_1 (u) + \text{powers of } t^{\geq 2} \big) \\
=~& Tu + t \big(  T \rho^L (x_0, u) + T_1 (u) + [x_0, Tu]_\mathfrak{g} \big) + \text{powers of } t^{\geq 2}.
\end{align*}
It follows from the above expression that the coefficients of $t$ vanishes. By repeating the above argument, we get that $T_t$ is equivalent to $T$.
\end{proof}

\medskip

\noindent {\bf Finite order deformations.} Let $T : \mathfrak{h} \rightarrow \mathfrak{g}$ be a $\lambda$-weighted relative Rota-Baxter operator. For a fixed $N \in \mathbb{N}$, consider the spaces $\mathfrak{g}[[t]]/(t^{N+1})$ and $\mathfrak{h}[[t]]/(t^{N+1})$. Both of them are $\mathbf{k}[[t]]/(t^{N+1})$. Moreover, by $\mathbf{k}[[t]]/(t^{N+1})$-bilinearity, the Leibniz structures on $\mathfrak{g}$ and $\mathfrak{h}$ can be extended to $\mathfrak{g}[[t]]/(t^{N+1})$ and $\mathfrak{h}[[t]]/(t^{N+1})$, respectively. Further, $\mathfrak{h}[[t]]/(t^{N+1})$ is a Leibniz $\mathfrak{g}[[t]]/(t^{N+1})$-representation.

\begin{defn}
An order $N$ deformation of $T$ is given by a finite polynomial of the form $T_t^N = \sum_{i=0}^N t^i T_i \in \mathrm{Hom}(\mathfrak{h}, \mathfrak{g})[[t]]/ (t^{N+1})$ with $T_0 = T$, such that the $\mathbf{k}[[t]]/(t^{N+1})$-linear map $T_t^N : \mathfrak{h}[[t]]/(t^{N+1}) \rightarrow \mathfrak{g}[[t]]/(t^{N+1})$ is a $\lambda$-weighted relative Rota-Baxter operator.
\end{defn}

Thus, $T_t^N = \sum_{i=0}^N t^i T_i$ is an order $N$ deformation of $T$ if and only if the identity (\ref{def-eqns}) holds for $n = 0, 1, \ldots, N$. These can be equivalently written as
\begin{align}\label{fin-eq}
d_T (T_n) = - \frac{1}{2} \sum_{ \substack{i+j = n\\ i, j \geq 1}} \llbracket T_i, T_j \rrbracket, ~\text{for } n = 0, 1, \ldots, N.
\end{align}

\begin{defn}(Obstruction cochain) Let $T_t^N = \sum_{i=0}^N t^i T_i$ be an order $N$ deformation of $T$. We define a $2$-cochain $Ob_{T_t^N} \in C^2_T (\mathfrak{h}, \mathfrak{g})$, called the {\bf obstruction cochain}, by
\begin{align}
Ob_{T_t^N} = - \frac{1}{2} \sum_{\substack{i+j = N+1\\ i, j \geq 1}} \llbracket T_i, T_j \rrbracket.
\end{align}
\end{defn}

\begin{prop}
The obstruction cochain $Ob_{T^N_t}$ is a $2$-cocycle in the cohomology complex of $T$, i.e. $\delta_\mathrm{Leib} (Ob_{T^N_t}) = 0$ (equivalently, $d_T (Ob_{T^N_t}) = 0$).
\end{prop}

\begin{proof}
We have
\begin{align*}
d_T ( - \frac{1}{2} \sum_{ \substack{ i+j = N+1\\ i, j \geq 1}} \llbracket T_i, T_j \rrbracket ) 
&=  - \frac{1}{2} \sum_{ \substack{  i+j = N+1 \\ i, j \geq 1}}   ( d ~\llbracket T_i, T_j \rrbracket + \llbracket T, \llbracket T_i, T_j \rrbracket \rrbracket ) \\
&= - \frac{1}{2} \sum_{\substack{  i+j = N+1\\ i, j \geq 1}}    \big( \llbracket d T_i, T_j \rrbracket -  \llbracket T_i, dT_j \rrbracket  + \llbracket \llbracket T, T_i \rrbracket, T_j \rrbracket  - \llbracket T_i, \llbracket T, T_j \rrbracket \rrbracket \big) \\
&=  - \frac{1}{2} \sum_{\substack{  i+j = N+1\\ i, j \geq 1}}    \big( \llbracket d_T T_i, T_j \rrbracket -  \llbracket T_i, d_T T_j \rrbracket \big)  \\
&= \frac{1}{4} \sum_{\substack{  i_1 +  i_2 + j = N+1\\ i_1, i_2, j \geq 1}} \llbracket  \llbracket T_{i_1}, T_{i_2} \rrbracket, T_j \rrbracket - \frac{1}{4} \sum_{\substack{  i + j_1 + j_2 = N+1\\ i, j_1, j_2 \geq 1}} \llbracket T_i, \llbracket T_{j_1}, T_{j_2} \rrbracket \rrbracket  \quad (\text{from } (\ref{fin-eq}))\\
&= \frac{1}{2} \sum_{\substack{  i + j + k = N+1\\ i, j, k \geq 1}} \llbracket  \llbracket T_{i}, T_{j} \rrbracket, T_k \rrbracket = 0.
\end{align*}
This completes the proof.
This proves the result.
\end{proof}

It follows from the above result that we obtain a cohomology class $[Ob_{T^N_t}] \in H^2_T (\mathfrak{h}, \mathfrak{g})$ associated to the order $N$ deformation $T^N_t.$ This cohomology class is called the obstruction class.

\begin{defn}
An order $N$ deformation $T^N_t = \sum_{i=0}^N t^i T_i$ is said to be extensible if there exists a linear map $T_{N+1} : \mathfrak{h} \rightarrow \mathfrak{g}$ which makes $T_t^{N+1} = T^N_t + t^{N+1} T_{N+1}$ into an order $N+1$ deformation.
\end{defn}

The following result gives a necessary and sufficient condition for the extensibility of a finite order deformation in terms of its obstruction class.

\begin{thm}
An order $N$ deformation $T^N_t$ is extensible if and only if the corresponding obstruction class $[Ob_{T^N_t}]$ vanishes.
\end{thm}

\begin{proof}
Let $T^N_t = \sum_{i=0}^N t^i T_i$ be an extensible order $N$ deformation. Since there exists a linear map $T_{N+1}$ which makes $T^{N+1}_t = \sum_{i=0}^N t^i T_i$ into an order $N+1$ deformation, we have that $Ob_{T^N_t} = d_T (T_{N+1}).$ In other words, $Ob_{T^N_t}$ is a coboundary. Hence the cohomology class $[Ob_{T^N_t}]$ vanishes. The converse part follows by similar argument.
\end{proof}

\section{Post-Leibniz algebras}\label{sec-6}

In this section, we introduce post-Leibniz algebras as a non-skewsymmetric analogue of post-Lie algebras. We show that post-Leibniz algebras arise naturally from $\lambda$-weighted relative Rota-Baxter operators. Finally, we study some properties of post-Leibniz algebras.

\begin{defn}
A {\bf post-Leibniz algebra} is a quadruple $(\mathfrak{a},  \triangleleft, \triangleright, [~,~]_\mathfrak{a})$ consisting of a vector space $\mathfrak{a}$ together with three bilinear operations $ \triangleleft, \triangleright, [~,~]_\mathfrak{a} : \mathfrak{a} \otimes \mathfrak{a} \rightarrow \mathfrak{a}$ satisfying for $u, v, w \in \mathfrak{a}$,
\begin{align}
u \triangleleft [v , w]_\star =~& (u \triangleleft v) \triangleleft w + v \triangleright (u \triangleleft w),\label{post-l1}\\
u \triangleright (v \triangleleft w) =~& (u \triangleright v) \triangleleft w + v \triangleleft [u , w]_\star, \label{post-l2}\\
u \triangleright (v \triangleright w) =~& [u , v]_\star \triangleright w + v \triangleright (u \triangleright w), \label{post-l3}\\
u \triangleright [v,w]_\mathfrak{a} =~& [u \triangleright v, w]_\mathfrak{a} + [ v, u \triangleright w]_\mathfrak{a}, \label{post-l4}\\
[u. v \triangleright w]_\mathfrak{a} =~& [ u \triangleleft v, w]_\mathfrak{a} + v \triangleright [u, w]_\mathfrak{a}, \label{post-l5}\\
[u, v \triangleleft w]_\mathfrak{a} =~& [u,v]_\mathfrak{a} \triangleleft w + [v, u \triangleleft w]_\mathfrak{a}, \label{post-l6}\\
[u, [v,w]_\mathfrak{a}]_\mathfrak{a} =~& [[u,v]_\mathfrak{a}, w]_\mathfrak{a} + [v, [u, w]_\mathfrak{a}]_\mathfrak{a}, \label{post-l7}
\end{align}
where $[u , v]_\star = u \triangleleft v + u \triangleright v + [u, v]_\mathfrak{a}$.
\end{defn}

\begin{remark}
(i) The notion of post-Lie algebras was first introduced by Vallette \cite{vallette} in the operadic study of generalized partitioned posets. A post-Lie algebra is a triple $(\mathfrak{a}, \circ, [~,~])$ consisting of a vector space $\mathfrak{a}$ together with bilinear operations $\circ, [~,~] : \mathfrak{a} \otimes \mathfrak{a} \rightarrow \mathfrak{a}$ in which $[~,~]$ is skewsymmetric and satisfying the following identities
\begin{align*}
&(u \circ v) \circ w - u \circ (v \circ w) - (v \circ u) \circ w + v \circ (u \circ w) + [u,v] \circ w = 0,\\
&u \circ [v,w] = [u \circ v, w]+[v, u \circ w],\\
&[u, [v,w]] + [v, [w,u]] + [w, [u,v]] = 0 ~~~~~ (\text{Jacobi identity}).
\end{align*}
Next, let $(\mathfrak{a}, \triangleleft, \triangleright, [~,~]_\mathfrak{a})$ be a post-Leibniz algebra with the property that $u \triangleleft v = - v \triangleright u$ and $[u, v]_\mathfrak{a} = - [v,u]_\mathfrak{a}$ for all $u, v \in \mathfrak{a}$ (such post-Leibniz algebras are called `skewsymmetric'). In this case, it can be checked that $(\mathfrak{a}, \triangleright, [~,~]_\mathfrak{a})$ is a post-Lie algebra. Therefore, ordinary post-Leibniz algebras can be thought of as a non-skewsymmetric analogue of post-Lie algebras.

(ii) The notion of pre-Leibniz algebras was recently introduced in the study of weight zero relative Rota-Baxter operators on Leibniz algebras. More precisely, a pre-Leibniz algebra is a triple $(\mathfrak{a}, \triangleleft, \triangleright)$ consisting of a vector space $\mathfrak{a}$ together with bilinear operations $\triangleleft, \triangleright : \mathfrak{a} \otimes \mathfrak{a} \rightarrow \mathfrak{a}$ satisfying for $u, v, w \in \mathfrak{a}$,
\begin{align*}
u \triangleleft ( v \triangleleft w + v \triangleright w ) =~& ( u \triangleleft v ) \triangleleft w + v \triangleright (u \triangleleft w), \\
u \triangleright ( v \triangleleft w ) =~& ( u \triangleright v ) \triangleleft w + v \triangleleft ( u \triangleleft w + u \triangleright w),\\
u \triangleright ( v \triangleright w) =~& ( u  \triangleleft v + u \triangleright v) \triangleright w + v \triangleright ( u \triangleright w). 
\end{align*}
Therefore, it follows that any pre-Leibniz algebra $(\mathfrak{a}, \triangleleft , \triangleright)$ is a post-Leibniz algebra $(\mathfrak{a}, \triangleleft, \triangleright, [~,~]_\mathfrak{a} = 0)$.
\end{remark}

\begin{prop}
Let $(\mathfrak{a}, \triangleleft, \triangleright, [~,~]_\mathfrak{a})$ be a post-Leibniz algebra. Then $(\mathfrak{a}, [~,~]_\star)$ is a Leibniz algebra. This is called the total Leibniz algebra, denoted by $\mathfrak{a}_\mathrm{Tot}$.
\end{prop}

\begin{proof}
By adding the left hand sides of the identities (\ref{post-l1})-(\ref{post-l7}), we simply get that $[u, [v,w]_\star ]_\star$.  On the other hand, by adding the right hand sides of (\ref{post-l1})-(\ref{post-l7}), we get $[[u, v]_\star, w]_\star + [v, [u,w]_\star ]_\star$. Therefore, the Leibniz identity (\ref{leibniz-identity}) holds for the bracket $[~,~]_\star$.
\end{proof}

\begin{prop}\label{rrb-pl}
Let $\mathfrak{g}, \mathfrak{h}$ be two Leibniz algebras and $\mathfrak{h}$ be a Leibniz $\mathfrak{g}$-representation. Let $T: \mathfrak{h} \rightarrow \mathfrak{g}$ be a $\lambda$-weighted relative Rota-Baxter operator. Then $(\mathfrak{h}, \triangleleft, \triangleright, [~,~]^\lambda_\mathfrak{h})$ is a post-Leibniz algebra, where
\begin{align*}
u \triangleleft v = \rho^R (u, Tv), \quad u \triangleright v = \rho^L (Tu,v) ~~~~ \text{~~ and ~~ } ~~~~ [u, v]^\lambda_\mathfrak{h} = \lambda [u, v]_\mathfrak{h}, ~ \text{ for } u, v \in \mathfrak{h}.
\end{align*}
\end{prop}

\begin{proof}
For any $u, v, w \in \mathfrak{h}$, we have
\begin{align*}
\rho^R (u, [Tv, Tw]_\mathfrak{g}) = \rho^R (u, T [v, w]_\star) = u \triangleleft [v,w]_\star.
\end{align*}
On the other hand,
\begin{align*}
\rho^R (\rho^R (u, Tv), Tw) + \rho^L (Tv, \rho^R (u, Tw)) = (u \triangleleft v) \triangleleft w + v \triangleright (u \triangleleft w).
\end{align*}
Since $(\rho^L, \rho^R)$ satisfies (\ref{rep3}), it follows that the right hand sides of the above equations are same. Hence (\ref{post-l1}) holds. Similarly, we observe that
\begin{align*}
&\rho^L (Tu, \rho^R (v, Tw)) = u \triangleright (v \triangleleft w),\\
&\rho^R (\rho^L (Tu, v), Tw) + \rho^R (v, [Tu, Tw]_\mathfrak{g}) = (u \triangleright v) \triangleleft w + v \triangleleft [u, w]_\star.
\end{align*}
Hence it follows from (\ref{rep2}) that the identity (\ref{post-l2}) holds. We also have
\begin{align*}
&\rho^L (Tu, \rho^L (Tv, w)) = u \triangleright (v \triangleright w),\\
& \rho^L ([Tu, Tv]_\mathfrak{g}, w) + \rho^L (Tv, \rho^L (Tu, w)) = [u, w]_\star \triangleright w + v \triangleright (u \triangleright w).
\end{align*}
Therefore, the identity (\ref{post-l3}) also holds. By the same way, we observe
\begin{align*}
&\rho^L (Tu, \lambda [v,w]_\mathfrak{h}) = u \triangleright [v,w]^\lambda_\mathfrak{h},\\
& \lambda [\rho^L (Tu, v), w]_\mathfrak{h} + \lambda [v, \rho^L (Tu, w)]_\mathfrak{h} = [u \triangleright v, w]_\mathfrak{h}^\lambda + [ v, u \triangleright w]_\mathfrak{h}^\lambda.
\end{align*}
It follows from (\ref{lrep-1}) that the identity (\ref{post-l4}) also holds. Similarly, we get the identities (\ref{post-l5}), (\ref{post-l7}) as we have (\ref{lrep-2}), (\ref{lrep-3}). Finally, the identity (\ref{post-l7}) automatically holds for the bracket $[~,~]_\mathfrak{h}^\lambda$ as the bracket $[~,~]_\mathfrak{h}$ satisfies the same. This completes the proof.
\end{proof}
 
In the previous proposition, we show that a $\lambda$-weighted relative Rota-Baxter operator induces a post-Leibniz algebra structure. Next, we prove the converse: any post-Leibniz algebra is always induced by a $1$-weighted relative Rota-Baxter operator.

Let $(\mathfrak{a}, \triangleleft , \triangleright , [~,~]_\mathfrak{a})$ be a post-Leibniz algebra. Consider the {total Leibniz algebra} $\mathfrak{a}_\mathrm{Tot} = (\mathfrak{a}, [~,~]_\star)$. We define maps $\varrho^L : \mathfrak{a}_\mathrm{Tot} \otimes \mathfrak{a} \rightarrow \mathfrak{a}$ and $\varrho^R : \mathfrak{a} \otimes \mathfrak{a}_\mathrm{Tot} \rightarrow \mathfrak{a}$ by
\begin{align*}
\varrho^L (u, v) = u \triangleright v ~~~ \text{ and } ~~~ \varrho^R (v, u) = v \triangleleft u, ~ \text{ for } u \in \mathfrak{a}_\mathrm{Tot}, v \in \mathfrak{a}.
\end{align*}
Then it can be checked that $\varrho^L, \varrho^R$ makes the Leibniz algebra $\mathfrak{a} = (\mathfrak{a}, [~,~]_\mathfrak{a})$ into a Leibniz $\mathfrak{a}_\mathrm{Tot}$-representation. With this notation, the identity map $\mathrm{id} : \mathfrak{a} \rightarrow \mathfrak{a}_\mathrm{Tot}$ is a $1$-weighted relative Rota-Baxter operator. Moreover, the induced post-Leibniz algebra structure on the vector space $\mathfrak{a}$ coincides with the given one.

\medskip

Given a Leibniz algebra, the following result gives a necessary and sufficient condition to have a compatible post-Leibniz algebra structure.

\begin{prop}
Let $(\mathfrak{g}, [~,~]_\mathfrak{g})$ be any Leibniz algebra. Then there is a compatible post-Leibniz algebra structure if and only if there exists a Leibniz $\mathfrak{g}$-representation $\mathfrak{h}$ and an invertible $1$-weighted relative Rota-Baxter operator $T: \mathfrak{h} \rightarrow \mathfrak{g}$.
\end{prop}

\begin{proof}
Suppose $(\mathfrak{g}, [~,~]_\mathfrak{g})$ has a compatible post-Leibniz algebra structure given by $(\mathfrak{g}, \triangleleft, \triangleright, [~,~]'_\mathfrak{g})$, i.e. the quadruple $(\mathfrak{g}, \triangleleft, \triangleright, [~,~]'_\mathfrak{g})$ is a post-Leibniz algebra and $[~,~]_\mathfrak{g} = \triangleleft + \triangleright + [~,~]'_\mathfrak{g}$. As discussed above, the Leibniz algebra $(\mathfrak{g}, [~,~]'_\mathfrak{g})$ is a Leibniz $\mathfrak{g}$-representation with left and right actions given by
\begin{align*}
\varrho^L (x,y) = x \triangleright y ~~~~ \text{ and } ~~~~ \varrho^R (x,y) = x \triangleleft y, ~ \text{ for } x, y \in \mathfrak{g}.
\end{align*}
Then the identity map $\mathrm{id} : \mathfrak{g} \rightarrow \mathfrak{g}$ (which is invertible) is a $1$-weighted relative Rota-Baxter operator. The induced Leibniz algebra structure on $\mathfrak{g}$ is given by $(\mathfrak{g}, [~,~]_\mathfrak{g})$.

Conversely, let $\mathfrak{h}$ be a Leibniz $\mathfrak{g}$-representation and $T: \mathfrak{h} \rightarrow \mathfrak{g}$ be an invertible $1$-weighted relative Rota-Baxter operator. We know from Proposition \ref{rrb-pl} that $\mathfrak{h}$ carries a post-Leibniz algebra structure. Using the invertibility of $T$, we get a post-Leibniz algebra structure on $\mathfrak{g}$ which is given by
\begin{align*}
x \triangleleft y = T (\rho^R (T^{-1} x, y)), ~~~~ x \triangleright y = T (\rho^L (x, T^{-1} y)) ~~~~ \text{ and } ~~~~ [x,y]'_\mathfrak{g} = T [T^{-1}x, T^{-1}y]_\mathfrak{h}, ~ \text{ for } x, y \in \mathfrak{g}.
\end{align*}
Moreover, we have
\begin{align*}
x \triangleleft y + x \triangleright y + [x, y]'_\mathfrak{g} =~& T \big(  \rho^R (T^{-1} x, y) + \rho^L (x, T^{-1} y) + [T^{-1}x, T^{-1}y]_\mathfrak{h}  \big) \\
=~& [TT^{-1} x, TT^{-1} y]_\mathfrak{g} = [x,y]_\mathfrak{g}.
\end{align*}
Hence $(\mathfrak{g}, \triangleleft, \triangleright, [~,~]'_\mathfrak{g})$ is a compatible post-Leibniz algebra structure for the Leibniz algebra $(\mathfrak{g}, [~,~]_\mathfrak{g})$.
\end{proof}

\medskip

\noindent {\bf Acknowledgements.} The author would like to thank IIT Kharagpur (India) for providing the beautiful academic atmosphere where the research has been carried out.

\noindent {\bf Data availability statement.} Data sharing not applicable to this article as no datasets were generated or analysed during the current study.

\medskip


%


\begin{thebibliography}{BFGM03}





\bibitem{bai-operad} C. Bai, O. Bellier and L. Guo, Splitting of operations, Manin products, and Rota-Baxter operators, {\em International Mathematics Research Notices} 2013 (2013) 485--524.

\bibitem{bala-operad} D. Balavoine, Deformations of algebras over a quadratic operad, {\em Operads: Proceedings of Renaissance Conferences (Hartford, CT/Luminy, 1995)}, 207-234, Contemp. Math., 202, {\em Amer. Math. Soc., Providence, RI,} 1997.


\bibitem{bloh}
A. Bloh, A generalization of the concept of a Lie algebra, {\em Dokl. Akad. Nauk SSSR} 165 (3) (1965) 471-473.

\bibitem{cartier} Cartier, P.: On the structure of free Baxter algebras. Adv. Math. 9 (1972) 253--265.

\bibitem{livernet} F. Chapoton and M. Livernet, Pre-Lie algebras and the rooted trees operad, {\em International Mathematics Research Notices}, Volume 2001, Issue 8 (2001) 395-408.

\bibitem{connes}
A. Connes and D. Kreimer, Renormalization in quantum field theory and the Riemann-Hilbert problem. I. The Hopf algebra structure of graphs and the main theorem.
{\em Comm. Math. Phys.} 210 (2000) 249--273.





\bibitem{das-rota} A. Das, Deformations of associative Rota-Baxter operators, {\em J. Algebra} 560 (2020) 144-180.












\bibitem{gers} M. Gerstenhaber, On the deformation of rings and algebras, {\em Ann. of Math. (2)} 79 (1964), 59-103.















\bibitem{loday-une} J.-L. Loday, Une version non commutative des alg\`{e}bres de Lie: les alg\`{e}bres de Leibniz, {\em Enseign. Math. (2)} 39 (1993) no. 3-4, 269-293.

\bibitem{loday-pira} J.-L. Loday and T. Pirashvili, Universal enveloping algebras of Leibniz algebras and (co)homology, {\em Math. Ann.} 296 (1993) no. 1, 139-158.








\bibitem{rota} G.-C. Rota, Baxter algebras and combinatorial identities, I, II, {\em Bull. Amer. Math. Soc. 75 (1969), 325-329; ibid.} 75 1969 330-334.

\bibitem{sheng-pl} Y. Sheng and R. Tang, Leibniz bialgebras, relative Rota-Baxter operators and the classical Leibniz Yang-Baxter equation, {\em J. Noncommutative Geom.} to appear.





\bibitem{tang} Tang, R.,  Bai, C., Guo, L., Sheng, Y.: Deformations and their controlling cohomologies of $\mathcal{O}$-operators. {Comm. Math. Phys.} 368 (2),  665--700 (2019).

\bibitem{sheng-defr} R. Tang, Y. Sheng and Y. Zhou, Deformations of relative Rota-Baxter operators on Leibniz algebras, {\em Int. J. Geom. Methods Mod. Phys.} Vol. 17, No. 12, 2050174 (2020).

\bibitem{uchino}  K. Uchino, Quantum analogy of Poisson geometry, related dendriform algebras and Rota-Baxter operators, {\em Lett. Math. Phys.} 85 (2008), no. 2-3, 91-109. 

\bibitem{vallette} B. Vallette, Homology of generalized partition posets, {\em J. Pure Appl. Algebra} 208 (2007) 699-725.

\bibitem{voro} Th. Voronov, Higher derived brackets and homotopy algebras, {\em J. Pure Appl. Algebra} 202 (2005), no. 1-3, 133-153.





\end{thebibliography}
\end{document}